\documentclass[a4paper,11pt]{article}

\usepackage[T1]{fontenc}
\usepackage{enumerate}
\usepackage[leqno]{amsmath}%
\usepackage{amsthm}
\usepackage{amsxtra}%
\usepackage{amsfonts}%
\usepackage{amssymb}%
\usepackage[margin=1in]{geometry}
\usepackage{color,hyperref}
\usepackage{graphicx}
\usepackage{tikz}
\usepackage{subfig}
\usepackage{cite}
\hypersetup{colorlinks,breaklinks,
             linkcolor=blue,urlcolor=blue,
             anchorcolor=blue,citecolor=blue}
\usepackage{color}
\usepackage{array}
\usepackage{booktabs}
\usepackage{dsfont}
\usepackage[toc,page]{appendix}
\usepackage{authblk}
\usepackage{booktabs}
\usepackage[font=small]{caption}
\usepackage{microtype}
\usepackage{nicefrac}
\usepackage{times}

\newcommand{\Cr}{\lambda}

\newcommand{\Lip}{\text{Lip}}
\newcommand{\range}{\text{range}}

\newcommand{\one}{\mathds{1}}

\renewcommand{\H}{\mathcal{H}}

\newcommand{\B}{{\mathcal B}}

\renewcommand{\O}{{\mathcal O}}
\renewcommand{\bar}[1]{{\overline{#1}}}

\newcommand{\R}{\mathbb{R}}

\newcommand{\eps}{\varepsilon}

\renewcommand{\phi}{\varphi}

\def\XXint#1#2#3{{\setbox0=\hbox{$#1{#2#3}{\int}$ }
\vcenter{\hbox{$#2#3$ }}\kern-.6\wd0}}

\newtheorem{theorem}{Theorem}
\newtheorem{lemma}[theorem]{Lemma}

\theoremstyle{definition}

\newtheorem{remark}[theorem]{Remark}

\numberwithin{equation}{section}
\numberwithin{theorem}{section}

\graphicspath{{./images/}}

\begin{document} 

\title{Asymptotically optimal strategies for online prediction with history-dependent experts}
\author{Jeff Calder\thanks{{\bf Funding:} Jeff Calder was supported by NSF-DMS grant 1944925 and the Alfred P. Sloan foundation.} }
\author{Nadejda Drenska}
\affil{School of Mathematics \\ University of Minnesota\thanks{{\bf Emails} \textit{\{jcalder,ndrenska\}@umn.edu}}}

\date{}

\maketitle

\begin{abstract}
We establish sharp asymptotically optimal strategies for the problem of online prediction with \emph{history dependent experts}. The prediction problem is played (in part) over a discrete graph called the $d$ dimensional \emph{de Bruijn graph}, where $d$ is the number of days of history used by the experts. Previous work \cite{drenska2019PDE} established $O(\eps)$ optimal strategies for $n=2$ experts and $d\leq 4$ days of history, while \cite{drenska2020Online} established $O(\eps^{1/3})$ optimal strategies for all $n\geq 2$ and all $d\geq 1$, where the game is played for $N$ steps and $\eps=N^{-1/2}$. In this paper, we show that the optimality conditions over the de Bruijn graph correspond to a graph Poisson equation, and we establish $O(\eps)$ optimal strategies for all values of $n$ and $d$.
\end{abstract}

\section{Introduction}

\emph{Prediction with expert advice} refers to problems in online machine learning \cite{CBL} where a player synthesizes advice from many experts to make predictions in real-time, often against an adversarial environment. The seminal work in the field is due to Cover \cite{cover1966behavior} and Hannan \cite{Hannan}, and since then, the field has grown substantially. We refer to \cite{ABG, CBL, GPS, HKW, LW, CFH, Ro} for effective algorithms that work well in practice, but may not be provably optimal, and to recent work \cite{drenska2020Online,drenska2017pde,drenska2020prediction,drenska2019PDE,  kobzar2020new1,   kobzar2020new2, ABG, GPS, Bayraktar} that has started to address asymptotic optimality in a rigorous way. Applications of prediction from expert advice include algorithm boosting \cite{FS}, stock price prediction and portfolio optimization \cite{FS}, self-driving car software \cite{AKT}, and many other problems.

In this paper, we study the problem of \emph{prediction of a binary sequence} with history dependent experts, which was studied recently in \cite{drenska2020Online,drenska2019PDE}, and is a generalization of Cover's original work \cite{cover1966behavior}. In this problem, we have a stream of binary data $b_1,b_2,b_3,\dots$ with $b_i\in \B:=\{-1,1\}$ and $n$ experts making predictions about $b_i$. We view the problem as a \emph{stock prediction problem}, with $b_i$ representing the increase or decrease of the price of a stock each day. The $n$ experts make their predictions using the previous $d$ days of market history
\begin{equation}\label{eq:mk}
m^i:=(b_{i-d},b_{i-d+1},\dots,b_{i-1}) \in \B^d.
\end{equation}
We denote the expert predictions by
\begin{equation}\label{eq:experts}
q_1,\dots,q_n:\B^d\to [-1,1],
\end{equation}
where $q_j(m)$ represents the prediction of expert $j$ given stock price history $m\in \B^d$. We write $q(m) = (q_1(m),q_2(m),\dots,q_n(m))$ for convenience. The investor's goal is to combine the expert advice to make their own prediction $f_i\in [-1,1]$, and the investor gains $b_if_i$ on day $i$, while the $j^{\rm th}$ expert would gain $b_iq_j(m^i)$ on day $i$, were they to invest their prediction.

The investor's performance is measured by their regret with respect to each expert, which is the difference between the gains of the expert and those of the player. Thus, on day $i$, the investor accumulates regret of $b_i(q_j(m^i)-f_i)$ with respect to expert $j$. After playing the game for $N$ days, the investor's final regret is evaluated with a payoff function $g:\R^n\to \R$, which is commonly taken as the regret with respect to the best performing expert, i.e., $g(x)=\max\{x_1,\dots,x_n\}$.  We assume there is an adversarial market controlling the binary data stream. The market observes the investor's choice $f_i$ before deciding on $b_i$, and both players have full knowledge of the expert strategies. The market's goal is to maximize the payoff, while the investor's goal is to minimize the payoff, yielding a two-player zero-sum game.

The possible transitions for the stock history $m^i$ are described by a directed graph with nodes $\B^d$ called the $d$-dimensional \emph{de Bruijn graph} over $2$ symbols. To describe the graph, we introduce the notation $m|b=(m_2,\dots,m_d,b)$ for concatenation of symbols, and use the short form $m_+=m|1$ and $m_-=m|-1$. With this notation, the stock history $m^i$ satisfies $m^{i+1}=m^i|b_i$. The de Bruijn graph has node set $\B^d$, and directed edges from $m$ to $m_+$ and from $m$ to $m_-$ for each $m\in \B^d$. The graph is depicted for $d=3$ in Figure \ref{fig:graph} with $0$ replacing $-1$ for convenience. The two-player game is played (in part) over the discrete de Bruijn graph, and this must be accounted for in the analysis of optimal strategies.\\
\begin{figure}
\begin{center}
\begin{tikzpicture}[scale=1.5]
\node at (-1,0) {010};
\node at (1,0) {101};
\node at (-3,0) {000};
\node at (3,0) {111};
\node at (-2,-1) {100};
\node at (-2,1) {001};
\node at (2,-1) {110};
\node at (2,1) {011};
\draw[thick,->] (-3.2,0.15) arc (25:335:0.3);
\draw[thick,->] (3.2,-0.15) arc (-180+25:180-25:0.3);
\draw[thick,->] (-2.8,0.1)--(-2.2,0.85);
\draw[thick,<-] (-2,0.85)--(-2,-0.85);
\draw[thick,->] (-1.8,0.85)--(-1.2,0.1);
\draw[thick,<-] (-1.8,-0.85)--(-1.2,-0.1);
\draw[thick,<-] (-2.8,-0.1)--(-2.2,-0.85);
\draw[thick,<-] (2.8,0.1)--(2.2,0.85);
\draw[thick,->] (2.8,-0.1)--(2.2,-0.85);
\draw[thick,->] (2,0.85)--(2,-0.85);
\draw[thick,<-] (1.8,0.85)--(1.2,0.1);
\draw[thick,->] (1.8,-0.85)--(1.2,-0.1);
\draw[thick,->] (-1.8,1)--(1.8,1);
\draw[thick,<-] (-1.8,-1)--(1.8,-1);
\draw[thick,->] (0.85,0.15) arc (90-25:90+25:2);
\draw[thick,->] (-0.85,-0.15) arc (270-25:270+25:2);
\end{tikzpicture}
\end{center}
\caption{The de Bruijn Graph, $d=3$}
\label{fig:graph}
\end{figure}
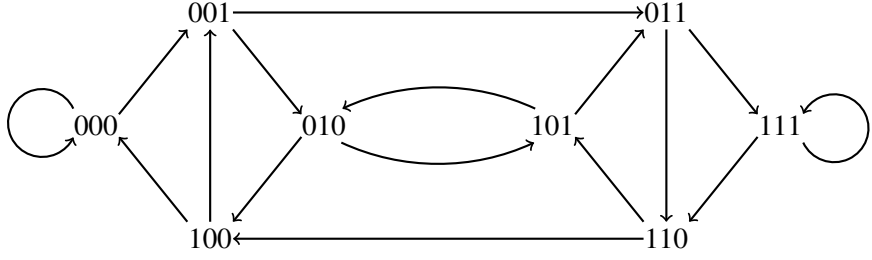
\indent We are interested in asymptotically optimal strategies for the investor and market when the game is played over a large number of turns $N$. In \cite{drenska2019PDE}, the second author and R.V.~Kohn studied investor strategies of the form
\begin{equation}\label{eq:ansatz}
f_i^* = \frac{\langle \nabla u,q(m)\rangle}{\langle \nabla u,\one\rangle} + \frac{\eps f_i^\#(m)}{2\langle \nabla u,\one\rangle},
\end{equation}
where $u(x,t)$ is the solution of a continuum PDE (see Eq.~\eqref{eq:PDE}), $\one$ is the ones vector, and $\eps=N^{-1/2}$. The first term in $f_i^*$ is a weighted average of the expert predictions, and arises from standard arguments based on Taylor expansions. The second term involving $f^\#_i(m)$ is more interesting, and accounts for the fluctuations in the strategy and value function over the de Bruijn graph. The authors of \cite{drenska2019PDE} showed that $f^\#_i(m)$ should be chosen by solving linear programs for the market and investor over the de Bruijn graph. When the optimal values of the linear programs for the market and investor agree, the strategy is provably $O(\eps)$ optimal. The authors of \cite{drenska2019PDE} solved these linear programs explicitly for $n=2$ experts and $d\leq 4$ days of history, and showed that their values agreed, leading to asymptotically optimal strategies in these cases. For $n=2$ and $d>4$, the authors gave upper and lower bounds for the value function.  A recent follow-up paper \cite{drenska2020Online} took a completely different approach, and established $O(\eps^{1/3})$ asymptotically optimal strategies for all $n\geq 2$ and all $d\geq 1$. These strategies have the form \eqref{eq:ansatz}, except that the correctors $f^\#_i$ grow to $\infty$ as $N\to \infty$, which is connected to the worse $O(\eps^{1/3})$ optimality. 

In this paper, we fully resolve the problems raised in \cite{drenska2019PDE}, and show in general that the correctors $f^\#_i$ are given by 
\begin{equation}\label{eq:inv_opt}
f^{\#}_i(m)=\H(m_+) - \H(m_-),
\end{equation}
where $\H:\B^d\to \R$ is the solution of a Poisson equation over the de Bruijn graph (for precise details, we refer the reader to Section \ref{sec:main}). The corrector \eqref{eq:inv_opt} appears in \cite[Eq.~(1.15)]{drenska2019PDE} as part of a far more complicated investor strategy, which is what led us to this solution. We show in this paper that the strategy \eqref{eq:ansatz} with $f^\#_i$ given by \eqref{eq:inv_opt} is $O(\eps)$ optimal as $N\to \infty$. The corresponding asymptotically optimal market strategy simply penalizes the investor when they deviate from $f^*$, that is, $b^*_i=\text{sign}(f_i^* - f_i)$. In Section \ref{sec:LP} we show that the correctors \eqref{eq:inv_opt} solve the linear programs identified in \cite{drenska2019PDE}, and we show how to solve the linear programs on more general directed graphs. This means that, in principle, we could invoke the proofs from \cite{drenska2019PDE} to establish our main result. However, our new insights connecting $f^\#_i$ to a graph Poisson equation lead to much simpler proofs of optimality.  

We briefly mention that there are many other cases in the PDE literature where scaling limits of two-player games lead to elliptic or parabolic PDEs (see, e.g., \cite{KS1,KS2,PS1,PS2,AS1,NS,APSS,LM,calder2020convex}). To our knowledge, this is the first case where the limiting equation is a PDE in continuous variables $(x,t)$ coupled to a PDE on a graph (the de Bruijn graph). The ideas in our paper are simple and we expect they may apply to other straightforward generalizations of this setup. We leave this to future work. We describe our results and give the proofs in Section \ref{sec:main}, and discuss the connections to the market and investor linear programs from \cite{drenska2019PDE}, and their generalizations to other directed graphs, in Section \ref{sec:LP}.

\section{Asymptotically optimal strategies}
\label{sec:main}

Our main result concerns the asymptotic behavior of the \emph{value function} $V_N(x,\ell;m)$, which represents the optimal value of the two-player game over $N$ steps, given the game starts on day $\ell\geq 1$ with regret $x\in \R^n$ and market history $m\in \B^d$. The value function is given by $V_N(x,N;m)=g(x)$ and 
\begin{equation}\label{eq:valuedef}
V_N(x,\ell;m) = \min_{|f_{\ell}|\leq 1}\max_{b_{\ell}=\pm 1}\cdots \min_{|f_{N-1}|\leq 1}\max_{b_{N-1}=\pm 1}g\left( x + \sum_{i=\ell}^{N-1} b_i(q(m^i) - f_i\one) \right),
\end{equation}
for $1 \leq \ell \leq N-1$, where $m^\ell = m$ and $m^{i+1}=m^i|b_i$ for $i=\ell,\dots,N-1$.  We assume $g\in C^4(\R^n)$ with uniformly bounded derivatives of order up to $4$ over $\R^n$,  there exists $\theta>0$ such that
\begin{equation}\label{eq:gone}
\nabla g(x) \cdot \one \geq \theta \ \ \text{ for all }x\in \R^n,
\end{equation}
and that $g$ is positively $1$-homogeneous, that is
\begin{equation}\label{eq:homogeneous}
g(sx) = sg(x) \ \ \text{for all }x\in \R^n, s>0.
\end{equation}
We also assume the expert strategies $q=(q_1,\dots,q_n)$ satisfy 
\begin{equation}\label{eq:expert}
q:\B^d \to [-\mu,\mu]^n \ \ \text{ for some }\mu \in (0,1).
\end{equation}

To understand the long time behavior of the value function, we define the parabolic rescaled version
\begin{equation}\label{eq:uN}
u_N(x,t;m) := \frac{1}{\sqrt{N}}V_N(\sqrt{N}x,\lceil Nt\rceil;m),
\end{equation}
where $\lceil t\rceil$ denotes the smallest integer larger than or equal to $t$. We also set $\eps=N^{-1/2}$ throughout the paper. The rescaled value function $u_N$ satisfies the dynamic programming principle (see \cite[Proposition 2.3]{drenska2020Online})
\begin{equation}\label{eq:dpp}
u_N(x,t;m) = \min_{|f|\leq 1}\max_{b=\pm 1}u_N(x + \eps b(q(m)-\one f),t+\eps^2;m|b).
\end{equation}
Under the assumptions above, it was shown in \cite{drenska2020Online} that $u_N$ converges uniformly as $N\to \infty$ at a rate of $\O(\eps^{1/3})$ to the linear growth solution of the diffusion equation
\begin{equation}\label{eq:PDE}
\left\{\begin{aligned}
u_t + \frac{1}{2^{d+1}}\sum_{\eta\in Q(\nabla u)} \langle \nabla^2 u\,\eta,\eta\rangle &= 0,&&\text{in }\R^n\times (0,1)\\ 
u &=g,&&\text{on }\R^n\times \{t=1\},
\end{aligned}\right.
\end{equation}
where
\begin{equation}\label{eq:Qdef}
Q(\nabla u)=\left\{q(m) - \frac{\langle \nabla u, q(m)\rangle}{\langle \nabla u, \one\rangle}\one \, : \, m\in \B^d\right\}.
\end{equation}
The same convergence result was established earlier in \cite{drenska2019PDE} for $n=2$ and $d\leq 4$ with the sharper $O(\eps)$ rate. The diffusion equation \eqref{eq:PDE} describes the limiting behavior of the rescaled value function $u_N$ and encodes information about asymptotically optimal strategies. 

The equation \eqref{eq:PDE} is a degenerate diffusion equation, since $Q(\nabla u)\subset \nabla u^\perp$. Here $\nabla u^\perp$ is the set of vectors, whose inner product with $\nabla u$ is $0$. Normally, solutions of such equations are not classical and must be interpreted in the viscosity sense. In this case, there is a hidden geometric structure in the equation that allows us to show (see \cite[Theorem 4.14]{drenska2020Online}) that the solution $u(x,t)$ is classical, and furthermore 
\begin{equation}\label{eq:uclassical}
\left\{\begin{aligned}
&\text{All spatial derivatives of $u$ up to order 4, and time derivatives}\\
&\text{up to order $2$ are uniformly bounded on $\R^n\times [0,1]$.}
\end{aligned}\right.
\end{equation}
The solution $u$ also has linear growth as $|x|\to \infty$ and satisfies (see \cite[Theorem 4.4]{drenska2020Online})
\begin{equation}\label{eq:uone}
\nabla u(x,t) \cdot \one \geq \theta \ \ \text{ for all }(x,t)\in \R^n\times [0,1].
\end{equation}

To describe the optimal investor strategy, we define
\[\xi(x,t;m) = q(m) - \frac{\langle \nabla u(x,t),q(m)\rangle}{\langle \nabla u(x,t),\one\rangle}\one,\]
and
\begin{equation}\label{eq:h}
h(x,t;m) =\frac{1}{2}\langle \nabla^2u(x,t)\, \xi(x,t;m),\xi(x,t;m)\rangle.
\end{equation}
We define the function $\H(x,t;m)$ by 
\begin{equation}\label{eq:H}
\H(x,t;m) = h(x,t;m) + \sum_{\ell=1}^{d-1} \frac{1}{2^{\ell}}\sum_{s\in \B^\ell}h(x,t;m|s), 
\end{equation}
where the notation $m|s$ for $m\in \B^d$ and $s\in \B^j$ refers to concatenation of symbols, defined by 
\[m|s = (m_{j+1},m_{j+2},\dots,m_d,s_1,s_2,\dots,s_j),\]
when $j< d$, and 
\[m|s = (s_{j-d+1},s_{j-d+2},\dots,s_j).\]
when $j \geq d$. Notice that $m|s$ is the state on the de Bruijn graph that we arrive at by starting from $m\in \B^d$ and following the edges defined by $s_1,\dots,s_j$. Thus, $\H$ can be interpreted as a weighted average of $h$ over a de Bruijn tree rooted at $m$. We note that by \eqref{eq:uclassical}, $\H$ is uniformly Lipschitz continuous on $\R^n\times [0,1]$, for any $m\in \B^d$. That is, there exists a constant $C>0$, depending on $u$, $\theta$, and $n$, but independent of $d$, such that
\begin{equation}\label{eq:HLip}
|\H(x,t;m) - \H(y,s;m)| \leq C d(|x-y| + |s-t|)
\end{equation}
holds for all $(x,t),(y,s)\in \R^n\times [0,1]$. 

Let us focus on the asymptotically optimal strategies for the investor and for the market.
The $O(\eps)$ asymptotically optimal investor strategy is given by
\begin{equation}\label{eq:fstar}
f^*(x,t;m) = \frac{\langle \nabla u(x,t),q(m)\rangle}{\langle \nabla u(x,t),\one\rangle } + \frac{\eps}{2}\left( \frac{\H(x,t;m_+) - \H(x,t;m_-)}{\langle \nabla u(x,t),\one\rangle} \right),
\end{equation}
and the asymptotically optimal market strategy is given by
\begin{equation}\label{eq:bstar}
b^*(x,t;m,f) = 
\begin{cases}
1,&\text{if }f \leq f^*(x,t;m)\\
-1,&\text{if }f > f^*(x,t;m).
\end{cases}
\end{equation}
That is, the market simply penalizes the player for deviating from $f^*$.

Using the strategies $f^*$ and $b^*$ defined above, we prove the following theorem.
\begin{theorem}\label{thm:main}
Let $g\in C^4(\R^n)$ with uniformly bounded derivatives of order up to $4$, and assume \eqref{eq:gone}, \eqref{eq:homogeneous}, and \eqref{eq:expert} hold. Let $u$ be the solution of \eqref{eq:PDE}. Then there exists $C_1,C_2>0$, depending on $u$, $n$ and $\theta$, such that  
\begin{equation}\label{eq:rate}
|u_N(x,t;m) - u(x,t)| \leq C_1d(1-t+\eps)\eps
\end{equation}
holds for all $N\geq C_2d^2/\mu^2$, $(x,t)\in \R^n\times [0,1]$ and $m\in \B^d$, where $\eps=N^{-1/2}$.
\end{theorem}
The proof of Theorem \ref{thm:main}, which is given in Section \ref{sec:proofs}, employs the strategy $f^*$ for the investor, while allowing the market to play optimally to obtain one direction of the rate. The other direction is obtained by playing the strategy $b^*$ for the market, and allowing the player to play optimally. This establishes that both the market and investor strategies $b^*$ and $f^*$, respectively, are $O(\eps)$ asymptotically optimal as $\eps\to 0$.

\begin{remark}
The $1$-homogeneity of $g$ is only used to show that $u_N(x,1)=g(x)$, and can be omitted if $u_N$ is defined instead by the rescaled value function definition
\[u_N(x,t;m) = \min_{|f_{\lceil Nt\rceil}|\leq 1}\max_{b_{\lceil Nt\rceil}=\pm 1}\cdots \min_{|f_{N-1}|\leq 1}\max_{b_{N-1}=\pm 1}g\left( x + \eps\sum_{i=\ell}^{N-1} b_i(q(m^i) - f_i\one) \right).\]
\label{rem:homo}
\end{remark}

\subsection{Proof of Theorem \ref{thm:main}}
\label{sec:proofs}

The proof of Theorem \ref{thm:main} is based on recognizing $\H$ as the solution of a graph Poisson equation over $\B^d$. We define the graph Laplacian $\Delta_{\B^d}$ on the de Bruijn graph by
\begin{equation}\label{eq:debruijnLap}
\Delta_{\B^d} v(m) = v(m) - \frac{1}{2}(v(m_+) + v(m_-))
\end{equation}
for any function $v:\B^d\to \R$. We have taken the convention from graph theory that $\Delta_{\B^d}$ is a positive definite operator, and hence has the opposite sign of the continuous Laplacian $\Delta u$. We also define the gradient $\nabla^b_{\B^d}$ on the de Bruijn graph by
\begin{equation}\label{eq:debruijnGrad}
\nabla_{\B^d}^b v(m) = v(m) - v(m|b).
\end{equation}
We write $\nabla_{\B^d}^\pm=\nabla_{\B^d}^{\pm 1}$ for convenience, and note that
\[\Delta_{\B^d} v(m) = \frac{1}{2}(\nabla_{\B^d}^+ v(m) + \nabla_{\B^d}^- v(m)).\]
We also denote the average of $v$ over the de Bruijn graph by 
\begin{equation}\label{eq:havg}
(v)_{\B^d} = \frac{1}{2^d}\sum_{m\in \B^d}v(m).
\end{equation}
In this notation, the PDE \eqref{eq:PDE} is equivalent to $u_t + (h)_{\B^d} = 0$, where we recall $h(x,t;m)$ is defined in \eqref{eq:h}. We will often drop the dependence on $(x,t)$ and $m$, when there is no confusion.

We now show that $\H$ solves a graph Poisson equation.
\begin{lemma}\label{lem:H}
For all $(x,t)\in \R^n\times [0,1]$ it holds that
\begin{equation}\label{eq:Poisson}
\Delta_{\B^d}\H = h - (h)_{\B^d} \ \ \text{ on }\B^d.
\end{equation}
Furthermore, $\H$ is the unique solution of \eqref{eq:Poisson}, up to a constant.
\end{lemma}
\begin{remark}
The equation \eqref{eq:Poisson} is a Poisson equation over the de Bruijn graph. The right hand side has mean value zero, which is a necessary and sufficient condition for the existence and uniqueness of a solution (up to a constant). In this light, the definition of $\H$ given in \eqref{eq:H} is a representation formula for the solution of the Poisson equation \eqref{eq:Poisson}, and can be thought of as convolution against the fundamental solution of Laplace's equation on the de Bruijn graph. Since the investor's strategy \eqref{eq:fstar} is the difference of $\H$ at $m_+$ and $m_-$, it is independent of the choice of constant in the solution of the Poisson equation \eqref{eq:Poisson}.
\label{rem:poisson}
\end{remark}
\begin{proof}
We compute
\begin{align*}
\Delta_{\B^d}\H(x,t;m)&=\frac{1}{2}\sum_{b\in \B}\left(\H(x,t;m) - \H(x,t;m|b)\right)\\
&= h(x,t;m) + \sum_{\ell=1}^{d-1} \frac{1}{2^{\ell}}\sum_{s\in \B^\ell}h(x,t;m|s)\\
&\hspace{1in}-\frac{1}{2}\sum_{b\in \B}h(x,t;m|b) -\frac{1}{2}\sum_{b\in \B}\sum_{\ell=1}^{d-1} \frac{1}{2^{\ell}}\sum_{s\in \B^\ell}h(x,t;m|b|s) \\
&= h(x,t;m) + \sum_{\ell=1}^{d-1} \frac{1}{2^{\ell}}\sum_{s\in \B^\ell}h(x,t;m|s) - \frac{1}{2}\sum_{\ell=0}^{d-1}\frac{1}{2^\ell}\sum_{s\in \B^{\ell+1}}h(x,t;m|s)\\
&=h(x,t;m) + \sum_{\ell=1}^{d-1} \frac{1}{2^{\ell}}\sum_{s\in \B^\ell}h(x,t;m|s) -\sum_{\ell=1}^{d}\frac{1}{2^{\ell}}\sum_{s\in \B^{\ell}}h(x,t;m|s) \\
&=h(x,t;m) -\frac{1}{2^d}\sum_{s\in \B^d}h(x,t;m|s).
\end{align*}
Since $m|s=s$ for $s\in \B^d$, we see that $\Delta_{\B^d}\H = h-(h)_{\B^d}$, as desired.

To see that $\H$ is unique, up to a constant, we use a maximum principle argument. Let $\bar{\H}:\B^d\to \R$ be any other solution of \eqref{eq:Poisson}, where $(x,t)$ is still fixed, and let $v = \bar{\H} - \H$. Then $\Delta_{\B^d} v(m)=0$ for all $m\in \B^d$. It follows that $v$ satisfies the mean value property
\begin{equation}\label{eq:MVP}
v(m) = \frac{1}{2}(v(m_+) + v(m_-))
\end{equation}
for all $m\in \B^d$. Let $m^*\in \B^d$ be any point where $v$ attains its maximum over $\B^d$. Then the mean value property \eqref{eq:MVP} implies that $v(m^*)=v(m^*_+) = v(m^*_-)$. Applying the same argument at $m^*_+$ and $m^*_-$ we obtain that $v(m^*)=v(m^*|s)$ for any $s\in \B^2$. We can continue this way to show that $v$ is constant on $\B^d$, which completes the proof.
\end{proof}

The proof of Theorem \ref{thm:main} also requires two lemmas showing how the solution $u$ of \eqref{eq:PDE} changes when either player plays their optimal strategy.
\begin{lemma}\label{lem:investor}
For $(x,t)\in \R^n\times [0,1-\eps^2]$, $m\in \B^d$, and any $b\in \B$, we have
\begin{equation}\label{eq:investor}
u(x + \eps b(q(m) - \one f^*),t+\eps^2) - u(x,t) = \nabla^b_{\B^d}\H(x,t+\eps^2;m)\eps^2 + \O(\eps^3),
\end{equation}
where $f^*=f^*(x,t+\eps^2;m)$.
\end{lemma}
\begin{proof}
Let $A$ denote the left hand side of \eqref{eq:investor}. By \eqref{eq:uclassical} we can Taylor expand to obtain
\[A=\eps^2 u_t + \eps b\langle \nabla u,q(m)-\one f^*\rangle + \frac{\eps^2}{2}\langle \nabla^2 u\, (q(m)-\one f^*),q(m)-\one f^*\rangle + \O(\eps^3),\]
where $u_t$, $\nabla u$ and $\nabla^2 u$ are evaluated at $(x,t+\eps^2)$ above. We now check that
\[q(m) - \one f^*(x,t+\eps^2;m) = \xi(x,t+\eps^2;m) + \O(\eps),\]
and
\[\langle \nabla u(x,t+\eps^2),q(m)-\one f^*(x,t+\eps^2;m)\rangle = -\frac{\eps}{2}\left( \H(x,t+\eps^2;m_+) - \H(x,t+\eps^2;m_-) \right).\]
Therefore
\[A =\left(u_t(x,t+\eps^2) +h(x,t+\eps^2;m) - \frac{b}{2}\left( \H(x,t+\eps^2;m_+) - \H(x,t+\eps^2;m_-) \right)\right) \eps^2  + \O(\eps^3).\]
By Lemma \ref{lem:H} we have $h = (h)_{\B^d} + \Delta_{\B^d}\H$. Inserting this above, and using the PDE \eqref{eq:PDE}, which is equivalent to $u_t + (h)_{\B^d}=0$, we obtain
\begin{align*}
A &= \left(\Delta_{\B^d}\H(x,t+\eps^2;m) - \frac{b}{2}\left( \H(x,t+\eps^2;m_+) - \H(x,t+\eps^2;m_-)  \right)\right)\eps^2 + \O(\eps^3)\\
&=\left( \H(x,t+\eps^2;m) -\frac{1+b}{2}\H(x,t+\eps^2;m_+) - \frac{1-b}{2}\H(x,t+\eps^2;m_-)  \right)\eps^2 + \O(\eps^3)\\
&=\left(\H(x,t+\eps^2;m) - \H(x,t+\eps^2;m|b) \right)\eps^2 + \O(\eps^3),
\end{align*}
which completes the proof.
\end{proof}
\begin{lemma}\label{lem:market}
For $(x,t)\in \R^n\times (0,1)$, $m\in \B^d$, and any $f\in [-1,1]$, we have
\begin{equation}\label{eq:market}
u(x + \eps b^*(q(m) - \one f),t+\eps^2) - u(x,t) \geq \nabla^{b^*}_{\B^d}\H(x,t+\eps;m)\eps^2 - C\eps^3,
\end{equation}
where $b^*=b^*(x,t+\eps^2;m,f)$.
\end{lemma}
\begin{proof}
We first note that for any $f\in [-1,1]$ we have
\[b^*(x,t+\eps^2;m,f)(f^*(x,t+\eps^2;m) - f) = |f^*(x,t+\eps^2;m)-f|.\]
Therefore $b^*(f^*-f) \geq 0$. By \eqref{eq:uone} we have $u(x+s\one,t) \geq u(x,t)$ for any $s\geq 0$, and so
\begin{align*}
u(x + \eps b^*(q(m) - \one f),t+\eps^2) &=u(x + \eps b^*q(m) - \eps\one b^*f^* + \eps\one b^*(f^*-f),t+\eps^2) \\
&\geq u(x + \eps b^*q(m) - \eps\one b^*f^*,t+\eps^2)\\
&=u(x + \eps b^*(q(m) - \one f^*),t+\eps^2).
\end{align*}
The proof is completed by applying Lemma \ref{lem:investor}.
\end{proof}

We now give the proof of Theorem \ref{thm:main}.

\begin{proof}
The proof is split into several steps.

1. Since $q(m)\in [-\mu,\mu]^n$ for all $m\in \B^d$ and $u$ satisfies \eqref{eq:uone} we have
\[\frac{\langle \nabla u(x,t),q(m)\rangle}{\langle \nabla u(x,t),\one\rangle}\in [-\mu,\mu].\]
We combine \eqref{eq:uone} with the fact that the difference $\H(x,t;m_+) - \H(x,t;m_-)$ is bounded by $Cd$ (because $\H$ is a weighted average of the bounded function $h$) to obtain
\[\frac{|\H(x,t;m_+) - \H(x,t;m_-)|}{\langle \nabla u(x,t),\one\rangle} \leq \frac{Cd}{\theta},\]
and so the strategy $f^*(x,t;m)$ given in \eqref{eq:fstar} is admissible, that is $f^*\in [-1,1]$, provided $\frac{Cd\eps}{2\theta} \leq \mu$. This is equivalent to $N\geq \frac{C^2d^2}{4\theta^2\mu^2}$ since $\eps=N^{-1/2}$.  For the remainder of the proof, we assume $N\geq \frac{C^2d^2}{4\theta^2\mu^2}$ so that $f^*\in [-1,1]$.

2.  
We claim there exists $K>0$, depending only on $u$, $n$, and $\theta$, such that for every $k\geq 1$ with $1-k\eps^2 \geq 0$ we have
\begin{equation}\label{eq:investor_bound}
u_N(x,1-k\eps^2;m) - u(x,1-k\eps^2) \leq \left(\H(x,1+\eps^2-k\eps^2;m) - \H^-(x)\right)\eps^2 + Kdk\eps^3,
\end{equation}
where $\H^-(x) = \min_{m\in \B^d}\H(x,1;m)$.
The proof is by induction. For the base case, we use \eqref{eq:dpp} to write
\begin{equation}\label{eq:dppfinal}
u_N(x,1-\eps^2;m) - u(x,1-\eps^2) =\min_{|f|\leq 1}\max_{b=\pm 1}\left\{u(x + \eps b(q(m)-\one f),1) - u(x,1-\eps^2)\right\}.
\end{equation}
Set $f=f^*(x,1;m)$ and apply Lemma \ref{lem:investor} to obtain 
\begin{align*}
u_N(x,1-\eps^2;m) - u(x,1-\eps^2) &\leq \max_{b=\pm 1}\nabla^b_{\B^d}\H(x,1;m)\eps^2 + C\eps^3\\
&\leq (\H(x,1;m) - \H^-(x))\eps^2 + C\eps^3.
\end{align*}
This establishes the base case, since $d\geq 1$. For the inductive step, assume \eqref{eq:investor_bound} holds for some $k\geq 1$, and use the dynamic programming principle \eqref{eq:dpp} and the inductive hypothesis to obtain  
\begin{align*}
&u_N(x,1-(k+1)\eps^2;m) - u(x,1-(k+1)\eps^2) \\
&= \min_{|f|\leq 1}\max_{b=\pm 1}\left\{u_N(x + \eps b(q(m)-\one f),1-k\eps^2;m|b) - u(x,1-(k+1)\eps^2)\right\}\\
&= \min_{|f|\leq 1}\max_{b=\pm 1}\Big\{u_N(x + \eps b(q(m)-\one f),1-k\eps^2;m|b) - u(x + \eps b(q(m)-\one f),1-k\eps^2)\\
&\hspace{1.2in} + u(x + \eps b(q(m)-\one f),1-k\eps^2)- u(x,1-(k+1)\eps^2)\Big\}\\
&\leq \min_{|f|\leq 1}\max_{b=\pm 1}\Big\{\left(\H(x + \eps b(q(m)-\one f),1+\eps^2-k\eps^2;m|b) - \H^-(x + \eps b(q(m)-\one f))\right)\eps^2\\
&\hspace{1.2in} + u(x + \eps b(q(m)-\one f),1-k\eps^2)- u(x,1-(k+1)\eps^2)\Big\} + Kdk\eps^3\\
&\leq \min_{|f|\leq 1}\max_{b=\pm 1}\Big\{\left(\H(x,1-k\eps^2;m|b) - \H^-(x)\right)\eps^2\\
&\hspace{1.2in} +u(x + \eps b(q(m)-\one f),1-k\eps^2)- u(x,1-(k+1)\eps^2)  \Big\} +Kdk\eps^3 + Cd\eps^3,
\end{align*}
where we used the Lipschitzness of $\H$ in the final step. We now set $f=f^*(x,1-k\eps^2;m)$ and apply Lemma \ref{lem:investor} to find that
\begin{align*}
&u_N(x,1-(k+1)\eps^2;m) - u(x,1-(k+1)\eps^2) \\
&\hspace{0.5in}\leq \max_{b=\pm 1}\Big\{ \H(x,1-k\eps^2;m|b) - \H^-(x) + \nabla^b_{\B^d}\H(x,1-k\eps^2;m)  \Big\}\eps^2 + Kdk\eps^3 + Cd\eps^3\\
&\hspace{0.5in}= \left(\H(x,1-k\eps^2;m) - \H^-(x)\right)\eps^2 + Kdk\eps^3 + Cd\eps^3.
\end{align*}
Choosing $K\geq C$ completes the proof by induction.

3. The proof is similar to part 2, so we sketch the important parts. We will show by induction that
\begin{equation}\label{eq:market_bound}
u_N(x,1-k\eps^2;m) - u(x,1-k\eps^2) \geq \left(\H(x,1+\eps^2-k\eps^2;m) - \H^+(x)\right)\eps^2 - Kdk\eps^3,
\end{equation}
where $\H^+(x) = \max_{m\in \B^d}\H(x,1;m)$. We set $b=b^*(x,1;m,f)$ in \eqref{eq:dppfinal} and apply Lemma \ref{lem:market} to obtain 
\[u_N(x,1-\eps^2;m) - u(x,1-\eps^2) \geq \nabla^{b^*}_{\B^d}\H(x,1;m)\eps^2 - C\eps^3,\]
which establishes the base case. For the inductive step, we follow the argument from part 2 to deduce
\begin{align*}
&u_N(x,1-(k+1)\eps^2;m) - u(x,1-(k+1)\eps^2) \\
&\geq \min_{|f|\leq 1}\max_{b=\pm 1}\Big\{\left(\H(x,1-k\eps^2;m|b) - \H^-(x)\right)\eps^2\\
&\hspace{1.2in} +u(x + \eps b(q(m)-\one f),1-k\eps^2)- u(x,1-(k+1)\eps^2)  \Big\} -Kdk\eps^3 - Cd\eps^3.
\end{align*}
We now set $b=b^*(x,1-k\eps^2;m,f)$ and apply Lemma \ref{lem:market} to find that
\begin{align*}
&u_N(x,1-(k+1)\eps^2;m) - u(x,1-(k+1)\eps^2) \\
&\hspace{0.5in}\geq \left(\H(x,1-k\eps^2;m|b^*) - \H^-(x) + \nabla^{b^*}_{\B^d}\H(x,1-k\eps^2;m)\right)\eps^2 - Kdk\eps^3 - Cd\eps^3\\
&\hspace{0.5in}= \left(\H(x,1-k\eps^2;m) - \H^-(x)\right)\eps^2 - Kdk\eps^3 - Cd\eps^3,
\end{align*}
which establishes the claim.

4. For any $t\in [0,1)$, we choose $k\geq 1$ so that
\[1-k\eps^2 \leq t < 1+\eps^2 - k\eps^2\]
and apply the results of parts 1 and 2.  Since $k\eps^2 < 1-t+\eps^2$, $|\H(x,t;m)|\leq Cd$, $u_N(x,1-k\eps^2;m) = u(x,t;m)$ and $|u(x,1-k\eps^2)-u(x,t)|\leq C\eps^2$, we obtain
\[|u_N(x,t;m) - u(x,t)| \leq Cd\eps^2 + Cd(1-t)\eps,\]
which completes the proof.
\end{proof}

\subsection{Lipschitz payoffs}

Our assumption that the payoff $g$ is $C^4$ in Theorem \ref{thm:main} precludes the commonly used payoff $g(x)=\max\{x_1,\dots,x_n\}$, which measures the regret with respect to the best performing expert, and is only a Lipschitz continuous payoff. Theorem \ref{thm:main} can be extended to Lipschitz payoffs, in a similar way as in \cite{drenska2020Online,drenska2019PDE,Zhu}, provided we place some additional assumptions on the payoff and expert strategies. We assume the payoff also satisfies the translation property
\begin{equation}\label{eq:translation}
g(x+s\,\one) = g(x) + s \ \ \text{for all }x\in\R^n \text{ and }s\in \R.
\end{equation}
We also make an assumption on diversity of expert strategies. We define $r:\B^d \to \R^{n-1}$ by
\begin{equation}\label{eq:rdef}
r(m)= (q_1(m)-q_n(m),\dots,q_{n-1}(m)-q_n(m)),
\end{equation}
and we assume
\begin{equation}\label{eq:ellipticity}
\frac{1}{2^{d+1}} \sum_{m\in \B^d} r(m)\otimes r(m) \geq \Cr I \ \ \text{for some } 0 < \Cr \leq 1,
\end{equation}
where $I$ is the $(n-1)\times (n-1)$ identity matrix. We recall that for symmetric matrices $A$ and $B$, the notation $A\geq B$ means that $A-B$ is positive semi-definite. 

Under these assumptions, the PDE \eqref{eq:PDE} exhibits parabolic smoothing, and even for Lipschitz payoffs $g$, the solution $u$ of \eqref{eq:PDE} is smooth. The condition \eqref{eq:ellipticity} can be interpreted exactly as the uniform ellipticity condition. We refer the reader to \cite[Theorem 4.12]{drenska2020Online}. While the smoothing is immediate when $t<1$, we lose the uniform estimates on the derivatives of $u$ as $t\to 1$. It was shown in \cite[Theorem 4.12]{drenska2020Online} that the derivatives of $u$ satisfy the estimates
\begin{equation}\label{eq:derivatives}
\left\{\begin{aligned}
|u_{\xi\xi}(x,t)| &\leq \frac{C\Lip(g)}{\sqrt{(1-t)\Cr}}, \ \ &|u_{\xi\xi\xi}(x,t)| \leq \frac{C\Lip(g)}{(1-t)\Cr},\\
|u_t(x,t)| &\leq \frac{C\Lip(g)}{\sqrt{1-t}}, \ \ \text{ and }\ \ &|u_{tt}(x,t)| \leq \frac{C\Lip(g)}{(1-t)^{3/2}},
\end{aligned}\right.
\end{equation}
for all $\xi\in \R^n$ with $|\xi|=1$ and all $(x,t)\in \R^n\times [0,1)$. Here, $\Lip(g)$ denotes the Lipschitz constant of $g$ and we use the notation $u_{\xi\xi}=\langle \nabla^2 u \xi,\xi\rangle$ and $u_{\xi\xi\xi} = \sum_{i,j,k=1}^{n} u_{x_ix_jx_k}\xi_i\xi_j\xi_k$.

Using this parabolic smoothing, we can prove the following result.
\begin{theorem}\label{thm:main2}
Let $g$ be Lipschitz continuous, and assume \eqref{eq:gone}, \eqref{eq:homogeneous}, \eqref{eq:expert}, \eqref{eq:translation}, and \eqref{eq:ellipticity} hold. Let $u$ be the solution of \eqref{eq:PDE}. Then there exists $C>0$, depending on $u$, $n$, $\theta$, and $\lambda$, such that for $N\gg 1$ we have that
\begin{equation}\label{eq:rate2}
|u_N(x,t;m) - u(x,t)| \leq Cd\eps\log|\eps|
\end{equation}
holds for all $(x,t)\in \R^n\times [0,1]$ and $m\in \B^d$, where $\eps=N^{-1/2}$.
\end{theorem}
The proof of Theorem \ref{thm:main2} follows closely that of Theorem \ref{thm:main}, except that we must keep track of how the error terms change due to the blow-up of the derivative estimates \eqref{eq:derivatives} as $t\to 1$. The proof is very similar to several results that were established in previous work (see \cite[Theorem 2]{Zhu}, \cite[Theorem 3.3]{drenska2019PDE}, \cite[Theorem 1.2]{drenska2020Online}), so we omit the details.

\section{The market and investor linear programs} 
\label{sec:LP}

We now make the connection between the optimal investor strategy $f^*$ defined in \eqref{eq:fstar} and the linear programs identified in \cite{drenska2019PDE}. Let us first recall from \cite{drenska2019PDE} how to arrive at the market and investor linear programs. We replace $u_N(x,t;m)$ by a smooth function $u(x,t)$  in the dynamic programming principle \eqref{eq:dpp} to obtain
\begin{equation}\label{eq:dppu}
\min_{|f|\leq 1}\max_{b=\pm 1}\left\{u(x + \eps b(q(m)-\one f),t+\eps^2) - u(x,t)\right\} = 0.
\end{equation}
We use the \emph{ansatz} \eqref{eq:ansatz} for the investor's choice $f$, which sets
\[\langle \nabla u,q(m)-f\one\rangle = -\frac{\eps}{2}f^\#.\]
Then we Taylor expand $u$ in \eqref{eq:dppu}, as in the proof of Lemma \ref{lem:investor}, and use the ansatz for $f$ to obtain
\begin{equation}\label{eq:dpp1step}
u_t + \min_{f^\#}\max_{b=\pm 1}\left\{h(m) - \frac{b}{2}f^\#(m)\right\} = O(\eps),
\end{equation}
where $h$ is defined in \eqref{eq:h}, and we are suppressing the dependence on $(x,t)$ everywhere. Of course, these computations are not rigorous, and not even approximately correct, since as written in \eqref{eq:dpp1step}, the optimal choice for the market would be $b=-\text{sign}(f^\#(m))$ and so the optimal investor's optimal choice would appear to be $f^\#=0$. This indicates that one step of the dynamic programming principle is insufficient, and the players will be choosing their strategies with more than one step of the game in mind. This idea was explored in \cite{drenska2020Online} where the authors used a $k$-step dynamic programming principle and took $k\to \infty$ as $N\to \infty$.

In contrast, the approach used in \cite{drenska2019PDE} considers multiple steps of \eqref{eq:dpp1step}, which results in accumulated regret of 
\begin{equation}\label{eq:kregret}
\sum_{i=1}^k \left(h(m^i) - \frac{b_i}{2}f^\#(m^i)\right)
\end{equation}
over $k$ steps, where $m^1,m^2,\dots,m^k\in \B^d$ is the path taken on the de Bruijn graph, which is determined by the market's choices of $b_i$ as $m^{i+1}=m^i|b_i$. Since the investor and market must be looking over multiple steps to choose their strategies, a reasonable approach for the investor is to choose $f^\#:\B^d\to \R$ so as to minimize \eqref{eq:kregret} over the worst case path $m^1,m^2,\dots,m^k$, which is what the market should select if they are playing optimally. To eliminate the market's choices $b_i$ from the problem, \cite{drenska2019PDE} used the fact that any closed walk on the de Bruijn graph can be decomposed into simple cycles, and instead minimized \eqref{eq:kregret} over all cycles on the de Bruijn graph. This leads to the investor's linear program, which is to find $M_{investor}\in \R$ and $f^\#:\B^d\to \R$ so as to minimize $M_{investor}$ subject to
\begin{equation}\label{eq:investorLP}
\frac{1}{|C|}\sum_{m\in C} \left(h(m) - \frac{b(m)}{2}f^\#(m)\right) \leq M_{investor}
\end{equation}
for all simple cycles $C$ on the de Bruijn graph, where $b(m)\in \B=\{-1,1\}$ denotes the outgoing edge from $m$ in the cycle. The corresponding market linear program is to maximize $M_{market}$ subject to
\begin{equation}\label{eq:marketLP}
\frac{1}{|C|}\sum_{m\in C} \left(h(m) - \frac{b(m)}{2}f^\#(m)\right) \geq M_{market}
\end{equation}
for all cycles $C$. The market and investor linear programs are not dual linear programs, though the setup has a similar flavor. It was shown in \cite{drenska2019PDE} (see Lemma 4.1 and 4.2) that the linear programs have solutions and 
\begin{equation}\label{eq:bounds}
M_{market}\leq (h)_{\B^d} \leq M_{investor}.
\end{equation}
The proof of this uses the fact that the de Bruijn graph is Eulerian, and so there is a closed walk on the graph visiting each edge exactly once. By explicitly solving the linear programs by hand, it was shown in \cite{drenska2019PDE} that $M_{market}=(h)_{\B^d}=M_{investor}$ for $n=2$ and $d\leq 4$, which the authors called achieving \emph{indifference}. We note that the linear programs also depend on $(x,t)$ and thus their solutions are changing (slowly) over the course of the game. When $n=2$, the dependence on $(x,t)$ can be factored out of linear program to simplify the situation (see \cite{drenska2019PDE}), though this seems to be possible only for $n=2$ experts.

We now show that the corrector $f^\#(m)=\H(m_+)-\H(m_-)$ used in this paper,  solves both the market and investor linear programs and achieves indifference. Indeed, using that  $\H$ solves the Poisson equation $\Delta_{\B^d}\H = h - (h)_{\B^d}$ we simply compute
\begin{align*}
h(m) - \frac{b(m)}{2}f^\#(m)&=h(m) - \frac{b(m)}{2}(\H(m_+)-\H(m_-))\\
&= (h)_{\B^d} + \Delta_{\B^d}\H(m)- \frac{b(m)}{2}(\H(m_+)-\H(m_-))\\
&=(h)_{\B^d} + \H(m) - \frac{1}{2}(\H(m_+) + \H(m_-))- \frac{b(m)}{2}(\H(m_+)-\H(m_-))\\
&=(h)_{\B^d} + \H(m) - \left(\frac{1+b(m)}{2}\right)\H(m_+) - \left(\frac{1-b(m)}{2}\right)\H(m_-)\\
&=(h)_{\B^d} + \H(m) - \H(m|b(m)).
\end{align*}
Notice the computation above is essentially the same as the main part of the proof of Lemma \ref{lem:investor}. When the quantity above is summed over a cycle on the de Bruijn graph, the terms $\H(m) - \H(m|b(m))$ contribute to a telescoping sum and exactly cancel out, yielding
\[\sum_{m\in C}\left(h(m) - \frac{b(m)}{2}f^\#(m)\right) = |C| (h)_{\B^d}\]
for any cycle $C$. We can use the choice $f^\#(m)=\H(m_+)-\H(m_-)$ in both the market and investor linear programs to show that $M_{market}\geq (h)_{\B^d}$ and $M_{investor}\leq (h)_{\B^d}$. Combining this with \eqref{eq:bounds} we find that $M_{market}=M_{investor}$ and so indifference is always achieved.

Given that we have found solutions of the linear programs from \cite{drenska2019PDE} achieving indifference, we could have used the same proof strategy as in \cite{drenska2019PDE}, or quoted their proofs directly. However, the observation that the solutions to the linear program produce a telescoping sum, as above, leads to substantial simplifications in the proofs. In particular, in the proof of Theorem \ref{thm:main} we have no need to consider cycles on the de Bruijn graph, and our market strategy is simpler than the one presented in \cite{drenska2019PDE}, which is split into several cases.

\subsection{General directed graphs}
\label{sec:general}

The solution $f^\#=\H(m_+)-\H(m_-)$ of the market and investor linear programs has a general form, depending on the solution of a Poisson equation over the de Bruijn graph. It is natural to ask whether these linear programs can be formulated and solved over more general directed graphs, or whether there is some structure in the de Bruijn graph that was essentially used in some way. In this section, we show that the ideas generalize quite naturally to other directed graphs, with the key requirements being that the graph admits an Eulerian cycle and the outgoing degree of each node is $2$. 

Let $G=(V,E)$ be a directed graph with vertex set $V$ consisting of $|V|$ nodes and edge set $E \subset V^2$. The edge set describes the set of directed edges, so each $e=(e_1,e_2) \in E$ defines a directed edge from $e_1\in V$ to $e_2\in V$. For each $x\in V$ we define the sets of incoming edges $I_x$ and outgoing edges $O_x$ by
\[I_x = \{e \in E \, : \, e_2 = x\} \ \ \text{ and }\ \ O_x=\{e \in E \, : \, e_1 = x\}.\]
We denote by $|I_x|$ and $|O_x|$ the cardinality of these sets.  A \emph{path} $P$ is a sequence of edges $e^1,e^2,\dots,e^m$ such that $e^i_2=e^{i+1}_1$ for $i=1,\dots,m-1$. The path connects $x\in V$ to $y\in V$ if $e^1_1=x$ and $e^m_2=y$. We say the graph is \emph{strongly connected} if there is a path connecting every $x,y\in V$. A \emph{cycle} $C$ is a path for which $e^{m}_2=e^1_1$. We denote by $|C|$ the number of edges in the cycle. An \emph{Eulerian cycle} is a cycle $C$ that traverses every edge $e\in E$ exactly once. If $G$ admits an Eulerian cycle, then it is strongly connected.

Let us denote by $\ell^2(V)$ the set of functions $u:V\to \R$. For a function $u\in\ell^2(V)$ we define the gradient
\begin{equation}\label{eq:gradV}
\nabla_V u(e) = u(e_1) - u(e_2)
\end{equation}
and the graph Laplacian
\begin{equation}\label{eq:LapV}
\Delta_V u(x) = \frac{1}{|O_x|}\sum_{e\in O_x}\nabla_V u(e).
\end{equation}
We also write the average of $u$ over $V$ as
\[(u)_V = \frac{1}{|V|} \sum_{x\in V }u(x).\]

We assume each node in the graph has two outgoing edges, so 
\begin{equation}\label{eq:two}
|O_x|=2 \ \ \text{ for all }x\in V.
\end{equation}
If the graph has an Eulerian cycle, then $|I_x|=|O_x|$ and so $|I_x|=2$ as well. In fact, by Euler's theorem, if $|I_x|=|O_x|$ for all $x\in V$ and $G$ is strongly connected, then $G$ has an Eulerian cycle. We also assume the two outgoing edges of each node are assigned $+1$ and $-1$ weights, as in the case of the de Bruijn graph. This defines a function $b:E \to \B$ that satisfies \begin{equation}\label{eq:bavg}
\sum_{e\in O_x}b(e) = 0 \ \ \text{ for all }x\in V.
\end{equation}
 We also assume we are given a function $h:V\to \R$, which can be interpreted as some base cost of traversing each node in the graph. The linear program generalizing the investor's linear program \eqref{eq:investorLP} is to find $f_1:V\to \R$ and $M_1$ to minimize $M_1$ subject to
\begin{equation}\label{eq:LP1}
 \frac{1}{|C|}\sum_{e\in C} \left(h(e_1) + b(e)f_1(e_1)\right) \leq M_1 \ \ \text{ for all cycles } C.
\end{equation}
The linear program generalizing the market's linear program \eqref{eq:marketLP} is to find $f_2:V\to \R$ and $M_2$ to maximize $M_2$ subject to 
\begin{equation}\label{eq:LP2}
 \frac{1}{|C|}\sum_{e\in C} \left(h(e_1) + b(e)f_2(e_1)\right) \geq M_2 \ \ \text{ for all cycles } C.
\end{equation}

In this section, we show that when $G$ admits an Eulerian cycle, the solutions to both linear programs \eqref{eq:LP1} and \eqref{eq:LP2} are given by
\begin{equation}\label{eq:sol}
f(x) = \frac{1}{2}\sum_{e\in O_x}b(e)\nabla_V \H(e),
\end{equation}
where $\H:V\to \R$ is the solution of the Poisson equation
\begin{equation}\label{eq:poissonGen}
\Delta_V \H = h - (h)_V \ \ \ \text{ on }V.
\end{equation}

The proof of this is straightforward, once we establish the existence of $\H$.
\begin{lemma}\label{lem:poissonV}
Assume $G$ admits an Eulerian cycle and \eqref{eq:two} holds. Then there exists $\H:V\to \R$ solving \eqref{eq:poissonGen}, and $\H$ is unique up to a constant.
\end{lemma}
\begin{proof}
First, we claim that the kernel of $\Delta_V$ is exactly the constant functions on $V$, that is
\begin{equation}\label{eq:ker}
\ker(\Delta_V) = \{u\in \ell^2(V) \, : \, u(x)=u(y) \text{ for all }x,y\in V\}.
\end{equation}
To see this, let $u\in \ell^2(V)$ with $\Delta_V u =0$. Then we have
\begin{equation}\label{eq:sumgrad}
\sum_{e \in E_x}\nabla_V u(e) = 0
\end{equation}
for all $x\in V$. Let $x_0\in V$ be a node where $u$ attains its maximum value over $V$. Then $\nabla_V u(e) \geq 0$ for all $e\in E_x$. Combining this with \eqref{eq:sumgrad} we have that $\nabla_V u(e)=0$ for all $e\in E_x$, and therefore $u$ also attains its maximum at all nodes $y\in V$ that are forward adjacent to $x$; that is all $y\in V$ such that $(x,y)\in E_x$. Since the graph $G$ is connected, we find that $u$ is constant on $V$, which establishes \eqref{eq:ker}.

We now claim that
\begin{equation}\label{eq:range}
\range(\Delta_V) =\{u\in \ell^2(V) \, : \, (u)_V =0\}. 
\end{equation}
To see this, we first note that by the rank-nullity theorem the dimension of $\range(\Delta_V)$ is $|V|-1$, since the kernel of $\Delta_V$ is one-dimensional. Therefore, we only need to show that 
\[\range(\Delta_V) \subset \{u\in \ell^2(V) \, : \, (u)_V =0\}.\]
to establish the claim, since the right hand side has dimension $|V|-1$. To prove this inclusion, let $u\in \ell^2(V)$ and compute
\begin{align*}
(\Delta_V u)_V &= \frac{1}{|V|}\sum_{x\in V}\Delta_V u(x)\\
&=\frac{1}{2|V|}\sum_{x\in V}\sum_{e\in E_x}\nabla_V u(e)\\
&=\frac{1}{2|V|}\sum_{e\in E}(u(e_1) - u(e_2))\\
&=\frac{1}{2|V|}\sum_{e\in E}u(e_1) - \frac{1}{2|V|}\sum_{e\in E}u(e_2)\\
&=\frac{1}{2|V|}\sum_{x\in V}|O_x|u(x) - \frac{1}{2|V|}\sum_{x\in V}|I_x|u(x)\\
&=\frac{1}{2|V|}\sum_{x\in V}(|O_x| - |I_x|)u(x).
\end{align*}
Since $G$ has an Eulerian cycle we have $|I_x|=|O_x|=2$, and so $(\Delta_v u)_V = 0$, which establishes the claim.

Since $h-(h)_V$ belongs to $\range(\Delta_V)$, there exists $\H$ such that $\Delta_V \H = h-(h)_V$. By the characterization of the kernel \eqref{eq:ker}, the solution is unique up to a constant.
\end{proof}

We now establish that the values of the linear programs agree and \eqref{eq:sol} is the solution to both linear programs.
\begin{theorem}\label{thm:gen}
Assume $G$ admits an Eulerian cycle, and that \eqref{eq:two} holds. Then \eqref{eq:sol} solves both linear programs \eqref{eq:LP1} and \eqref{eq:LP2}, and 
\begin{equation}\label{eq:}
M_1 = (h)_V = M_2.
\end{equation}
\end{theorem}
\begin{proof}
Let $f$ be given by \eqref{eq:sol} and note that
\begin{align*}
h(e_1) + b(e)f(e)&=(h)_V + \Delta_V \H(e_1) + \frac{b(e)}{2}\sum_{e'\in O_x}b(e')\nabla_V \H(e') \\
&=(h)_V + \frac{1}{2}\sum_{e\in O_x}\nabla_V \H(e) +\frac{1}{2}\sum_{e'\in O_x}b(e)b(e')\nabla_V \H(e') \\
&=(h)_V + \frac{1}{2}\sum_{e'\in O_x}(1 + b(e)b(e'))\nabla \H(e')\\
&=(h)_V + \nabla_V \H(e).
\end{align*}
Therefore, for any cycle $C$ we have
\[\sum_{e\in C}(h(e_1) + b(e)f(e)) = |C|(h)_V + \sum_{e\in C}\nabla_V \H(e) = |C|(h)_V.\]
It follows that $M_1 \leq (h)_V \leq M_2$.

To prove equality, let $C$ be an Eulerian cycle. Then for any $f:V\to \R$ we have
\begin{align*}
\sum_{e\in C}(h(e_1) + b(e)f(e_1))&=\sum_{x\in V}\sum_{e \in O_x}(h(x) + b(e)f(x))\\
&= 2\sum_{x\in V}h(e_1) + \sum_{x\in V} (f(x) - f(x)) = 2|V|(h)_V.
\end{align*}
Since $|C|=2|V|$ we see that
\[\frac{1}{|C|}\sum_{e\in C}(h(e_1) + b(e)f_1(e_1)) = (h)_V.\]
It follows that $M_1 \geq (h)_V$ and $M_2 \leq (h)_V$, which completes the proof.
\end{proof}

\begin{remark}
It seems non-trivial to generalize these ideas to more than a dichotomy of choices at each node. That is, the ideas in this section do not immediately generalize to $|O_x|\geq 3$, even if the outgoing and incoming degrees are the same at all nodes. We leave this as an interesting problem for future work.
\label{rem:extension}
\end{remark}

\section{Conclusion}
\label{sec:con}

We established $O(\eps)$ optimal strategies for online prediction with history dependent experts. The optimal strategies involve solving a Poisson equation over the de Bruijn graph, which should be interpreted as the corresponding Bellman (or Isaacs) equation in the discrete setting. We also checked that our optimal strategies solve the market and investor linear programs from \cite{drenska2019PDE}, which resolves the open problems from that paper, and we generalized the linear programs and their solutions to other directed graphs.

\bibliographystyle{abbrv}
\bibliography{ref}

\end{document}